\documentclass[a4paper, 11pt, reqno]{amsart}

\usepackage{amsmath,amssymb,amsthm}
\usepackage{graphicx}
\usepackage{mathrsfs}
\usepackage{enumerate}

\usepackage{verbatim}
\usepackage{amssymb}
\usepackage{amsbsy}
\usepackage{amscd}
\usepackage{amsmath}
\usepackage{amsthm}
\usepackage{mathrsfs}
\usepackage{eucal}

\newtheorem{theorem}{Theorem}\numberwithin{theorem}{section}
\newtheorem{prop}[theorem]{Proposition}

\newtheorem{lemma}[theorem]{Lemma}

\theoremstyle{remark}
\newtheorem{remark}[theorem]{Remark}

\numberwithin{equation}{section}

\def\at{\widetilde{a}}
\def\bt{\widetilde{b}}
\def\bb{B}

\def\ath{\mathrm{A}}
\def\C{\mathbb{C}}

\def\cc{C}

\def\kk{k_1,\ldots,k_r}

\long\def\comment#1{}

\def\Z{\mathbb{Z}}
\newcommand{\st}[2]{\left\{\begin{matrix} {#1}\\{#2} \end{matrix}\right\}}

\DeclareFontEncoding{OT2}{}{}
\DeclareFontSubstitution{OT2}{cmr}{m}{n}
\DeclareSymbolFont{cyss}{OT2}{wncyss}{m}{n}
\DeclareMathSymbol{\sh}{\mathbin}{cyss}{`x}

\allowdisplaybreaks

\title{On poly-cosecant numbers}

\author[M. Kaneko]{Masanobu Kaneko}
\author[M. Pallewatta]{Maneka Pallewatta}
\author[H. Tsumura]{Hirofumi Tsumura}

\address{M.\,Kaneko: Faculty of Mathematics, Kyushu University, Motooka 744, Nishi-ku, Fukuoka 819-0395, Japan}
\email{mkaneko@math.kyushu-u.ac.jp}

\address{M.\,Pallewatta: Graduate School of Mathematics, Kyushu University, Motooka 744, Nishi-ku, Fukuoka 819-0395, Japan}
\email{maneka.osh@gmail.com}

\address{{H.\,Tsumura:} Department of Mathematical Sciences, Tokyo Metropolitan University, 1-1, Minami-Ohsawa, Hachioji, Tokyo 192-0397, Japan}
\email{tsumura@tmu.ac.jp}



\begin{document}

\begin{abstract}
We introduce and study a ``level two'' generalization of the poly-Bernoulli numbers,
which may also be regarded as a generalization of the cosecant numbers.
We prove a recurrence relation, two exact formulas, and a duality relation for negative
upper-index numbers.
\end{abstract}

\maketitle

\section{Introduction} \label{sec-1}

Poly-Bernoulli numbers were first introduced in \cite{Kaneko1997} and later a slightly 
modified version was studied in \cite{AK1999}. They are, denoted $B_n^{(k)}$ and $C_n^{(k)}$ respectively, 
defined by using generating series, as follows. 
For an integer $k\in \mathbb{Z}$, 
let $\{\bb_n^{(k)}\}$ and $\{\cc_n^{(k)}\}$ be the sequences of rational numbers given respectively by
\begin{align}
&\frac{{\rm Li}_{k}(1-e^{-t})}{1-e^{-t}}=\sum_{n=0}^\infty \bb_n^{(k)}\frac{t^n}{n!} \label{eq-1-1}\\
\intertext{and}
&\frac{{\rm Li}_{k}(1-e^{-t})}{e^t-1}=\sum_{n=0}^\infty \cc_n^{(k)}\frac{t^n}{n!},  \label{eq-1-2}
\end{align}
where ${\rm Li}_{k}(z)$ is the polylogarithm function (or rational function when $k\le0$) defined by
\begin{equation}
{\rm Li}_{k}(z)=\sum_{m=1}^\infty \frac{z^m}{m^k}\quad (|z|<1). \label{eq-1-3}
\end{equation}
In the sequel, we regard this or any other series only as a formal power series.

Since ${\rm Li}_1(z)=-\log(1-z)$, the generating functions on the left-hand sides of \eqref{eq-1-1}
and \eqref{eq-1-2}   when $k=1$ become 
\[ \frac{te^t}{e^t-1} \quad \text{and} \quad \frac{t}{e^t-1} \] respectively, and hence 
$B_n^{(1)}$ and $C_n^{(1)}$
are usual Bernoulli numbers, the only difference being $B_1^{(1)}=1/2$ and $C_1^{(1)}=-1/2$
and otherwise $B_n^{(1)}=C_n^{(1)}$. 

Various properties of poly-Bernoulli numbers, including combinatorial applications, are known.
Among them we mention the explicit formulas
\[ B_n^{(k)} =(-1)^n\sum_{i=0}^n\frac{(-1)^ii!\st ni}{(i+1)^k},\quad
C_n^{(k)}=(-1)^n\sum_{i=0}^n\frac{(-1)^ii!{\st{n+1}{i+1}}}{(i+1)^k}\] 
for $k\in\Z,\,n\in \mathbb{Z}_{\geq 0}$ using the Stirling numbers of the second kind, and the dualities
\begin{align}
& B_n^{(-k)}=B_{k}^{(-n)}, \label{eq-1-4}\\
& C_n^{(-k-1)}=C_{k}^{(-n-1)} \label{eq-1-5}
\end{align}
for $k,n\in \mathbb{Z}_{\geq 0}$ 
(see \cite[Theorems\ 1\ and\ 2]{Kaneko1997} and \cite[\S\,2]{Kaneko-Mem}). For combinatorial
applications, see~\cite{Ben2017}.

In this paper, we study the following ``level $2$'' analog of poly-Bernoulli numbers, denoted~$D_n^{(k)}$,
which we also call the poly-cosecant numbers. For each $k\in\Z$, define $D_n^{(k)}$ by
\begin{equation}
\frac{\ath_k(\tanh(t/2))}{\sinh t}=\sum_{n=0}^\infty D_n^{(k)}\frac{t^n}{n!}, \label{eq-1-6}
\end{equation}
where $\ath_k(z)$ is the series 
\begin{equation}
\ath_k(z)=2\sum_{n=0}^\infty \frac{z^{2n+1}}{(2n+1)^k} \label{eq-1-7}
\end{equation}
and $\tanh(z)$ and $\sinh(z)$ are the usual hyperbolic tangent and sine functions respectively.
Since $\ath_k(z)$, $\tanh(z)$ and $\sinh(z)$ are all odd functions, we immediately see that 
$D_{2n+1}^{(k)}=0$ for all $n\in \mathbb{Z}_{\geq 0}$.
Note that $\ath_1(z)=2\tanh^{-1}(z)$, and thus 
\[ \sum_{n=0}^\infty D_n^{(1)}\frac{t^n}{n!} = \frac{t}{\sinh t} = \frac{i t}{\sin(i t)}\quad(i=\sqrt{-1} ).  \]
Hence, up to sign, $D_n^{(1)}$ is the cosecant number $D_n$ (see N\"orlund~\cite[p.~458]{Noe}).
 
We should mention that our $D_n^{(k)}$ is (if slightly modified) a special case of a generalization
of the poly-Bernoulli number introduced by Y.~Sasaki in \cite[Definition 5]{Sasaki2012}.
 
\section{Recurrence and explicit formulas for poly-cosecant numbers}

In this section, we  obtain a recurrence and explicit formulas for poly-cosecant numbers. 

We first give a recurrence.  Note that $D_0^{(0)}=1$ and $D_n^{(0)}=0$ for all $n\ge1$ 
because $\ath_0(\tanh(t/2))=\sinh(t)$.
Starting from this, the following formula gives a way to compute ${D}_{n}^{(k)}$ recursively
for any integer $k$.

\begin{prop}\label{Pr-3-1} For any integer $ k$  and $n\ge 0 $, it holds
\begin{equation*}
{D}_{n}^{(k-1)}=\sum_{m=0}^{\lfloor\frac{n}{2}\rfloor}\binom{n+1}{2m+1} {D}_{n-2m}^{(k)}.
\end{equation*}
\end{prop}

\begin{proof}
We differentiate the defining relation 
\begin{equation*}
\ath_k(\tanh(t/2))=\sinh t\sum_{n=0}^\infty  {D}_n^{(k)}\frac{t^n}{n!}
\end{equation*}
to obtain
\begin{align*}
\frac{ \ath_{k-1}(\tanh(t/2))}{\sinh t}&=\cosh t\sum_{n=0}^\infty {D}_n^{(k)}\frac{t^n}{n!}
+\sinh t\sum_{n=1}^\infty {D}_n^{(k)}\frac{t^{n-1}}{(n-1)!}. 
\end{align*}
From this we have
\begin{align*}
\sum_{n=0}^\infty {D}_n^{(k-1)}\frac{t^n}{n!}&=\sum_{m=0}^\infty\frac{t^{2m}}{(2m)!} 
\sum_{n=0}^\infty {D}_n^{(k)}\frac{t^n}{n!}+\sum_{m=0}^\infty\frac{t^{2m+1}}{(2m+1)!} 
\sum_{n=1}^\infty {D}_n^{(k)}\frac{t^{n-1}}{(n-1)!} \\
&=\sum_{n=0}^\infty\sum_{m=0} ^{\lfloor\frac{n}{2}\rfloor} {D}_{n-2m}^{(k)}\frac{t^{n}}{(2m)!(n-2m)!}
+\sum_{n=1}^\infty\sum_{m=0} ^{\lfloor\frac{n}{2}\rfloor} {D}_{n-2m}^{(k)}
\frac{t^{n}}{(2m+1)!(n-2m-1)!}\\
&=\sum_{n=0}^\infty\sum_{m=0} ^{\lfloor{\frac{n}{2}}\rfloor} \binom{n}{2m} {D}_{n-2m}^{(k)}
\frac{t^{n}}{n!}+\sum_{n=1}^\infty\sum_{m=0} ^{\lfloor{\frac{n}{2}}\rfloor} 
\binom{n}{2m+1} {D}_{n-2m}^{(k)}\frac{t^{n}}{n!}\\
&=\sum_{n=0}^\infty\sum_{m=0} ^{\lfloor{\frac{n}{2}}\rfloor} 
\binom{n+1}{2m+1} {D}_{n-2m}^{(k)}\frac{t^{n}}{n!}.
\end{align*}

By equating the coefficients of ${t^{n}}/{n!}$ on both sides, we obtain the desired result.
\end{proof}

When $k>0$, we may want to write this as
\[ (n+1) D_n^{(k)}=D_n^{(k-1)}-\sum_{m=1}^{\lfloor\frac{n}{2}\rfloor}\binom{n+1}{2m+1} {D}_{n-2m}^{(k)}
\quad (n>0). \]
Note that $D_0^{(k)}=1$ for all $k\in\Z$.

We proceed to give two explicit formulas for ${D}_{n}^{(k)}$.  Recall that $\begin{bmatrix}
n\\ m \end{bmatrix} $ and $\st{n}{m}$ are Stirling numbers of the first and the second kinds,
respectively, and $B_n=B_n^{(1)}$ is the Bernoulli number.  See \cite[Chapter~2]{AIK2014} for the
precise definition and formulas we use in the proof.  In~\cite{Sasaki2012}, Sasaki gave a different 
formula, but one needs to define yet another sequences to describe the formula.

\begin{theorem}\label{Th-3-1} 
For any $k\in \mathbb{Z}$ and $ n\ge 0$, we have

1)  \begin{align*}
       {D}_{n}^{(k)}&=4\sum_{m=0}^{\lfloor\frac{n}{2}\rfloor}\frac{1}{(2m+1)^{k+1}}\sum_{p=1}^{2m+1}
       \sum_{q=0}^{n-2m}(2^{p+q+1}-1)\binom{n}{q}\begin{bmatrix}
2m+1\\
p
\end{bmatrix} \st{n-q}{2m} \frac{B_{p+q+1}}{p+q+1},
\end{align*}

and

2) 
 \begin{align*}
  {D}_{n}^{(k)}=\sum_{m=0}^{\lfloor\frac{n}{2}\rfloor}\frac{1}{(2m+1)^{k+1}}
  \sum_{p=2m}^{n}\frac{(-1)^{p} (p+1)!}{2^{p}}\binom{p}{2m}\st{n+1}{p+1}.
\end{align*}
\end{theorem}
\begin{proof}
To prove 1), we need the following lemma. We may prove this in the same manner as 
in \cite[Proposition~2.6~(4)]{AIK2014}  and we omit the proof here.

\begin{lemma}\label{Lem-3-1} For $n\ge1$ we have,
\begin{equation*}
x^n \left(\frac{d}{dx}\right)^n=\sum_{m=1}^{n}(-1)^{n-m}\begin{bmatrix}
n\\
m
\end{bmatrix}\left(x\frac{d}{dx}\right)^m. 
\end{equation*}
\end{lemma}

We write

\begin{align}\label{eq-3-4}
\sum_{n=0}^{\infty}D_n^{(k)}\frac{t^n}{n!}&=\frac{\ath_{k}(\tanh (t/2))}{\sinh t}\nonumber \\
&=2 \sum_{m=0}^{\infty} \frac{(\tanh (t/2))^{2m+1}}{(2m+1)^k}\frac{1}{\sinh t}\nonumber\\
&=4 \sum_{m=0}^{\infty}\frac{1}{(2m+1)^k}\frac{e^t(e^t-1)^{2m}}{(e^t+1)^{2m+2}}.&&
\end{align}

Since
\begin{align} \label{eq-3-5}
\frac{1}{(x+1)^{n+1}}&=\frac{(-1)^n}{n!}\left(\frac{d}{dx}\right)^n\frac{1}{x+1},  &&
\end{align}
we see by setting $x=e^t$ and using Lemma~\ref{Lem-3-1} that
\begin{align} \label{eq-3-8}
\frac{e^{nt}}{(e^t+1)^{n+1}}&=\frac{1}{n!}\sum_{p=1}^{n}(-1)^{p} \begin{bmatrix}
n\\
p
\end{bmatrix}\left(\frac{d}{dt}\right)^p\frac{1}{e^t+1}.  &&
\end{align}

From
\begin{align*} 
\frac{t}{e^t-1}&=\sum_{q=0}^\infty  \bb_q \frac{t^q}{q!}
\end{align*}
and 
\begin{align*} 
\frac{1}{e^t+1}&=\frac{1}{e^t-1}-\frac{2}{e^{2t}-1},
\end{align*}
we have
\begin{align*} 
\frac{1}{e^t+1}&=\sum_{q=0}^\infty (1-2^q)\bb_q \frac{t^{q-1}}{q!}.
\end{align*}
By taking the $p$-th derivative of both sides, we get
\begin{align*} 
\left(\frac{d}{dt}\right)^p\left(\frac{1}{e^t+1}\right)&=\sum_{q=p+1}^\infty  (1-2^q)
\frac{\bb_q}q \frac{t^{q-p-1}}{(q-p-1)!}
=\sum_{q=p+1}^\infty  (1-2^{p+q+1})\frac{\bb_{p+q+1}}{p+q+1} \frac{t^{q}}{q!}
\end{align*}
and we substitute this in \eqref{eq-3-8} to obtain
\begin{align*} 
\frac{e^{nt}}{(e^t+1)^{n+1}}&=\frac1{n!}\sum_{p=1}^n (-1)^{p} \begin{bmatrix}
n\\
p
\end{bmatrix}\sum_{q=0}^\infty  (1-2^{p+q+1})\frac{\bb_{p+q+1}}{p+q+1} \frac{t^{q}}{q!}\\
&=\frac1{n!}\sum_{q=0}^\infty \sum_{p=1}^n (-1)^{p} \begin{bmatrix}
n\\
p
\end{bmatrix} (1-2^{p+q+1})\frac{\bb_{p+q+1}}{p+q+1} \frac{t^{q}}{q!}.
\end{align*}
From this, we have
\begin{align*} 
\frac{e^{t}}{(e^{t}+1)^{2m+2}}&=\frac{e^{-(2m+1)t}}{(e^{-t}+1)^{2m+2}}\\
&=\frac1{(2m+1)!}\sum_{q=0}^\infty \sum_{p=1}^{2m+1} (-1)^{p+q} \begin{bmatrix}
2m+1\\
p
\end{bmatrix} (1-2^{p+q+1})\frac{\bb_{p+q+1}}{p+q+1} \frac{t^{q}}{q!}.
\end{align*}
Together with the well-known generating series (\cite[Proposition~2.6~(7)]{AIK2014}, note that $\st{s}{2m}=0$ if
$s<2m$)
\begin{align*} 
(e^{t}-1)^{2m}&=(2m)!\sum_{s=0}^\infty \st{s}{2m} \frac{t^{s}}{s!},
\end{align*}
we obtain
\begin{align*}
&\frac{e^{t}(e^t-1)^{2m}}{(e^{t}+1)^{2m+2}}\\
&=\frac1{2m+1}\sum_{q=0}^\infty \sum_{s=0}^\infty\sum_{p=1}^{2m+1} (-1)^{p+q} (1-2^{p+q+1}) 
\begin{bmatrix}
2m+1\\
p
\end{bmatrix}\st{s}{2m}\frac{\bb_{p+q+1}}{p+q+1} \frac{t^{q+s}}{q!s!} \\
&=\frac{1}{2m+1}\sum_{n=0}^\infty\sum_{q=0}^n \sum_{p=1}^{2m+1} (-1)^{p+q} (1-2^{p+q+1})
\binom{n}{q} \begin{bmatrix}
2m+1 \\
p
\end{bmatrix}\st{n-q}{2m}\frac{ \bb_{p+q+1}}{p+q+1}\frac{t^n}{n!}. \nonumber \\
\end{align*}
Substituting this into \eqref{eq-3-4}, we have 
\begin{align*}
&\sum_{n=0}^{\infty}D_n^{(k)}\frac{t^n}{n!}\\
&=4 \sum_{m=0}^{\infty}\frac{1}{(2m+1)^{k+1}}\sum_{n=0}^\infty\sum_{q=0}^n 
\sum_{p=1}^{2m+1} (-1)^{p+q} (1-2^{p+q+1})\binom{n}{q} \begin{bmatrix}
2m+1 \\
p
\end{bmatrix}\st{n-q}{2m}\frac{ \bb_{p+q+1}}{p+q+1}\frac{t^n}{n!} \nonumber \\
&=4 \sum_{n=0}^\infty\sum_{m=0}^{\lfloor\frac{n}{2}\rfloor}\frac{1}{(2m+1)^{k+1}}
\sum_{p=1}^{2m+1}\sum_{q=0}^{n-2m} (2^{p+q+1}-1)
\binom{n}{q} \begin{bmatrix}
2m+1 \\
p
\end{bmatrix}\st{n-q}{2m}\frac{ \bb_{p+q+1}}{p+q+1}\frac{t^n}{n!}.
\end{align*}
(We have used the facts that $B_{p+q+1}=0$ if $p+q\ge1$ is even and $\st{n-q}{2m}=0$ if $n-q<2m$.)
By equating the coefficients of ${t^{n}}/{n!}$ on both sides, we obtain the desired result.

To prove 2), we employ the following formula (\cite[Proposition 9]{Cvi2011}) for the numbers $T_{n,m}$ 
(``higher order tangent numbers'') defined by
\begin{align}\label{eq-3-12}
\frac{\tan^m t}{m!}= \sum_{n=m}^{\infty} T_{n,m}\frac{t^n}{n!},
\end{align}
namely
\begin{align}\label{eq-3-13}
T_{n,m}=\frac{i^{n-m}}{m!}\sum_{p=m}^{n}(-2)^{n-p}p!\binom{p-1}{m-1} \st{n}{p}.
\end{align}

From the definition we have
\begin{align}\label{eq-3-14}
\sum_{n=0}^{\infty}D_n^{(k)}\frac{t^n}{n!}&=\frac{\ath_{k}(\tanh (t/2))}{\sinh t}
=\frac{d}{dt}\ath_{k+1}(\tanh (t/2))\nonumber \\
&=2\frac{d}{dt} \sum_{m=0}^{\infty}\frac{(\tanh (t/2))^{2m+1}}{(2m+1)^{k+1}}.&&
\end{align}

By using $\tanh t=-i\tan (it)$ and equations \eqref{eq-3-12} and \eqref{eq-3-13}, we can write
\begin{align*}
(\tanh (t/2))^m&=(-i)^m m!\sum_{n=m}^\infty T_{n,m}\frac{i^n}{2^n}\frac{t^n}{n!}\\
&=(-i)^m (-1)^{\frac{n-m}{2}} \sum_{n=m}^\infty \sum_{p=m}^{n}(-2)^{n-p} p!\binom{p-1}{m-1}
\st{n}{p}\frac{i^n}{2^n}\frac{t^n}{n!}\\
&=(-1)^m \sum_{n=m}^\infty \sum_{p=m}^{n}(-1)^{p}\frac{p!}{2^p} \binom{p-1}{m-1}\st{n}{p}\frac{t^n}{n!}.
\end{align*}
We therefore have
\begin{align*}
\sum_{n=0}^{\infty}D_n^{(k)}\frac{t^n}{n!}&=\sum_{m=0}^{\infty}\frac{1}{(2m+1)^{k+1}}
\sum_{n=2m+1}^\infty  
\sum_{p=2m+1}^{n}(-1)^{p+1} \frac{p!}{2^{p-1}}\binom{p-1}{2m}\st{n}{p}\frac{t^{n-1}}{(n-1)!}\\
&=\sum_{m=0}^{\infty}\frac{1}{(2m+1)^{k+1}}\sum_{n=2m}^\infty \sum_{p=2m}^{n}(-1)^{p} 
\frac{(p+1)!}{2^{p}}\binom{p}{2m}\st{n+1}{p+1}\frac{t^{n}}{{n}!}\\
&=\sum_{n=0}^{\infty}\sum_{m=0}^{\lfloor\frac{n}{2}\rfloor} \frac{1}{(2m+1)^{k+1}}  
\sum_{p=2m}^{n} \frac{(-1)^{p}(p+1)!}{2^{p}}\binom{p}{2m}\st{n+1}{p+1}\frac{t^{n}}{{n}!}.
\end{align*}
By equating the coefficients of ${t^{n}}/{n!}$, we complete the proof of the theorem.
\end{proof}

\section{Duality}\label{sec-2}

We now prove the duality property of $D_n^{(k)}$ similar to \eqref{eq-1-4} and \eqref{eq-1-5}.

\begin{theorem}\label{Th-2-1} For $n,k\in \mathbb{Z}_{\geq 0}$, it holds

\begin{equation}
D_{2n}^{(-2k-1)}=D_{2k}^{(-2n-1)}. \label{eq-2-1}
\end{equation}
\end{theorem}

We give two proofs using a generating function.  The first proof gives a closed, symmetric formula 
for the generating function, whereas the second is more indirect and a little involved.  We however
think the second way may be of independent interest and decided to include it here.

Consider the following generating function of $D_{2n}^{(-2k-1)}$:
\begin{equation*}
F(x,y)=\sum_{n=0}^\infty \sum_{k=0}^\infty D_{2n}^{(-2k-1)}\frac{x^{2n}}{(2n)!}\frac{y^{2k}}{(2k)!}.
\end{equation*}
We establish the closed formula of $F(x,y)$ as follows.  The theorem follows immediately from the 
symmetry of the formula.
\begin{prop}\label{C-2-7}  Set 
\[ G(x,y)=\frac{e^{x+y}}{(1+e^x+e^y-e^{x+y})^2}.\]  Then we have
$$F(x,y)=G(x,y)+G(x,-y)+G(-x,y)+G(-x,-y).$$
In other words, $F(x,y)$ is the sub-series of $4G(x,y)$ which is even both in $x$ and $y$.
\end{prop}

\begin{proof} 

We first compute the generating function of all $D_n^{(-k)}$,
\begin{equation}
f(x,y)=\sum_{n=0}^\infty \sum_{k=0}^\infty D_{n}^{(-k)}\frac{x^{n}}{n!}\frac{y^k}{k!}.  \label{eq-2-13}
\end{equation}

\begin{prop}\label{Pr-2-6} We have
\begin{equation}
f(x,y)=\frac{e^x(e^y-1)}{1+e^x+e^{y}-e^{x+y}}+\frac{e^{-x}(e^y-1)}{1+e^{-x}+e^y-e^{-x+y}}.\label{eq-2-14}
\end{equation}
\end{prop}

\begin{proof}
By definition
\begin{align*}
f(x,y)&=\sum_{k=0}^\infty \frac{\ath_{-k}(\tanh(x/2))}{\sinh x}\frac{y^k}{k!}\\
& =\frac{2}{\sinh x}\sum_{k=0}^\infty \sum_{n=0}^\infty (2n+1)^k (\tanh(x/2))^{2n+1}\frac{y^k}{k!}.
\end{align*}
We note that
\[ 2\sum_{n=0}^\infty (2n+1)^k t^{2n+1}=2\left(t\frac{d}{dt}\right)^{k}\frac{t}{1-t^2}
=\left(t\frac{d}{dt}\right)^{k}\left(\frac1{1-t}-\frac1{1+t}\right),\]
and by using the standard formula ({\it cf., e.g.,} \cite[Proposition~2.6~(4)]{AIK2014})
\[ \left(t\frac{d}{dt}\right)^{k}=\sum_{m=1}^{k}\st{k}{m}t^m\left(\frac{d}{dt}\right)^{m}, \]
we see the right-hand side is equal to
\begin{align*}
& \sum_{m=1}^{k}\st{k}{m}t^m\left(\frac{d}{dt}\right)^{m}\left(\frac{1}{1-t}-\frac{1}{1+t}\right)\\
& \ \ =\sum_{m=1}^{k}\st{k}{m} m!\left(\frac{t^m}{(1-t)^{m+1}}-\frac{(-t)^m}{(1+t)^{m+1}}\right).
\end{align*}
Hence, by setting $t=\tanh(x/2)$ and noting $t/(1-t)=(e^x-1)/2,\,-t/(1+t)=(e^{-x}-1)/2$, 
$(\sinh x) (1-t)=e^{-x}(e^x-1)$, $(\sinh x) (1+t)=e^x-1$, we have 
\begin{align*}
f(x,y)&=\frac{1}{\sinh x}\sum_{k=0}^\infty\sum_{m=1}^{k}\st{k}{m} m!
\left(\frac{t^m}{(1-t)^{m+1}}-\frac{(-t)^m}{(1+t)^{m+1}}\right)\frac{y^k}{k!}
\quad\  (t=\tanh(x/2))\\
& =\sum_{k=0}^\infty \sum_{m=1}^k \st{k}{m}m!\bigg\{ \frac{e^x}{e^x-1}\left(\frac{e^x-1}{2}\right)^m
-\frac{1}{e^x-1}\left(\frac{e^{-x}-1}{2}\right)^m\bigg\}\frac{y^k}{k!}\\
& =\sum_{m=1}^\infty (e^y-1)^m \bigg\{ \frac{e^x}{e^x-1}\left(\frac{e^x-1}{2}\right)^m
-\frac{1}{e^x-1}\left(\frac{e^{-x}-1}{2}\right)^m\bigg\}\\
& = \frac{e^x}{e^x-1}\cdot\frac{(e^y-1)(e^x-1)}{2-(e^y-1)(e^x-1)}-\frac{1}{e^x-1}
\cdot\frac{(e^y-1)(e^{-x}-1)}{2-(e^y-1)(e^{-x}-1)}\\
&=\frac{e^x(e^y-1)}{1+e^x+e^{y}-e^{x+y}}+\frac{e^{-x}(e^y-1)}{1+e^{-x}+e^y-e^{-x+y}}.
\end{align*}
\end{proof}
From \eqref{eq-2-14} we see that $f(x,y)$ is even in $x$, and so we have 
$$\frac{f(x,y)-f(x,-y)}{2}=\sum_{n=0}^\infty \sum_{k=0}^\infty 
D_{2n}^{(-2k-1)}\frac{x^{2n}}{(2n)!}\frac{y^{2k+1}}{(2k+1)!}.$$
Our generating function $F(x,y)$ is the derivative of this with respect to $y$, and 
Proposition~\ref{C-2-7} follows from a straightforward calculation.
Theorem~\ref{Th-2-1} is thus proved. 

\end{proof}

\begin{remark}\label{Rem-2-8}
We recall that
$$\sum_{n=0}^\infty \sum_{k=0}^\infty C_{n}^{(-k-1)}\frac{x^{n}}{n!}\frac{y^{k}}{k!}=
\frac{e^{x+y}}{(e^x+e^y-e^{x+y})^2}$$
(see \cite[Section 2]{Kaneko-Mem}), which is remarkably similar to $G(x,y)$. The general 
coefficients of $4G(x,y)$ not necessarily even either in $x$ or $y$ may worth studying.
The first several terms are given as
\begin{align*} 4G(x,y)&=1+\frac{x}{1!}+\frac{y}{1!}+\frac{x^2}{2!}+2\frac{x}{1!}\frac{y}{1!}
+\frac{y^2}{2!}+\frac{x^3}{3!}+4\frac{x^2}{2!}\frac{y}{1!}+4\frac{x}{1!}\frac{y^2}{2!}
+\frac{y^3}{3!}\\
&\quad +\frac{x^4}{4!}+8\frac{x^3}{3!}\frac{y}{1!}+13\frac{x^2}{2!}\frac{y^2}{2!}+
8\frac{x}{1!}\frac{y^3}{3!}+\frac{y^4}{4!}+\cdots.
\end{align*}
\end{remark}

For the second proof of Theorem~\ref{Th-2-1}, we need several lemmas.

\begin{lemma}\label{Lem-2-2}
\begin{equation*}
F(x,y)=2\sum_{n=0}^{\infty}\frac{\partial}{\partial x}\left(\tanh^{2n+1}(x/2)\right)\cosh((2n+1)y). 
\end{equation*}
\end{lemma}

\begin{proof}
By \eqref{eq-1-6}, we have
\begin{align*}
F(x,y)&= 2\sum_{k=0}^\infty \frac{\ath_{-2k-1}(\tanh(x/2))}{\sinh(x)}\frac{y^{2k}}{(2k)!}\\
&= \frac{2}{\sinh(x)}\sum_{k=0}^\infty \sum_{n=0}^\infty (2n+1)^{2k+1}\tanh^{2n+1}(x/2)\frac{y^{2k}}{(2k)!}\\
&= \frac{2}{\sinh(x)}\sum_{n=0}^\infty (2n+1)\tanh^{2n+1}(x/2)\cosh((2n+1)y)\\
&= \frac{1}{\sinh(x/2)\cosh(x/2)}\sum_{n=0}^\infty (2n+1)\tanh^{2n}(x/2)\frac{\sinh(x/2)}{\cosh(x/2)}\cosh((2n+1)y)\\
&=2\sum_{n=0}^{\infty}\frac{\partial}{\partial x}\left(\tanh^{2n+1}(x/2)\right)\cosh((2n+1)y).
\end{align*}
Thus we have the assertion.
\end{proof}

We write
$$F(x,y)=\sum_{m=0}^\infty g_m(x)\frac{y^{2m}}{(2m)!}
=\sum_{m=0}^\infty h_m(y)\frac{x^{2m}}{(2m)!}.$$
Then if we could prove $g_m(x)=h_m(x)$ for any $m\geq 0$, we are done. 

First, we look at  $g_m(x)$. Using Lemma~\ref{Lem-2-2}, we have
\begin{align*}
g_m(x)&=\left(\frac{\partial}{\partial y}\right)^{2m}F(x,y)\,\bigg|_{y=0}
=2\frac{d}{dx}\sum_{n=0}^\infty (2n+1)^{2m}\tanh^{2n+1}(x/2).
\end{align*}
Here we note that
\begin{equation}
\sum_{n=0}^\infty (2n+1)^{2m}t^{2n+1}=\left(t\frac{d}{dt}\right)^{2m}\sum_{n=0}^\infty t^{2n+1}
=\left(t\frac{d}{dt}\right)^{2m}\frac{t}{1-t^2}.\label{eq-2-3}
\end{equation}
Setting $t=\tanh(x/2)$ and noting
$$dt=\frac{1}{2}\frac{1}{\cosh^2(x/2)}dx,\quad \frac{t}{1-t^2}
=\frac{\tanh(x/2)}{1-\tanh^2(x/2)}=\frac{1}{2}\sinh x,$$
we have
$$t\frac{d}{dt}=\tanh(x/2)\cdot 2\cosh^2(x/2)\frac{d}{dx}=\sinh x\, \frac{d}{dx}.$$
Therefore we obtain
\begin{equation}
g_m(x)=\frac{d}{dx}\left(\sinh x\,\frac{d}{dx}\right)^{2m}\sinh x. \label{eq-2-4}
\end{equation}
We can explicitly write down the right-hand side by using the following lemma. 

For $m\in \mathbb{Z}_{\geq 0}$, we define sequences $\{ a_i^{(m)}\}_{0\leq i\leq m}\subset \mathbb{Q}$ 
inductively by 
\begin{equation}
\begin{split}
& a_0^{(0)}=1, \\
& a_i^{(m)}=\frac{1}{2}\left\{ i(2i-1)a_{i-1}^{(m-1)}-(2i+1)^2a_{i}^{(m-1)}+(i+1)(2i+3)a_{i+1}^{(m-1)}\right\}
\quad (m\geq 1), 
\end{split}
\label{eq-2-5}
\end{equation}
where we formally interpret $a_{i}^{(m)}=0$ for $i<0$ or $i>m$.

\begin{lemma}\label{Lem-2-3} For $m\in \mathbb{Z}_{\geq 0}$, 
\begin{equation}
\left(\sinh x\,\frac{d}{dx}\right)^{2m}\sinh x=\sum_{i=0}^{m}a_i^{(m)}\sinh((2i+1)x).  \label{eq-2-6}
\end{equation}
\end{lemma}

\begin{proof}
We give the proof by induction on $m$. For $m=0$, the identity trivially holds. We assume 
\begin{equation*}
\left(\sinh x\,\frac{d}{dx}\right)^{2(m-1)}\sinh x=\sum_{i=0}^{m-1}a_i^{(m-1)}\sinh((2i+1)x).  
\end{equation*}
Using 
$$\cosh(kx)\sinh(x)=\frac{1}{2}\left( \sinh((k+1)x)-\sinh((k-1)x)\right),$$
we have
\begin{align*}
& \left(\sinh x\,\frac{d}{dx}\right)^{2m-1}\sinh x=\frac{1}{2}\sum_{i=0}^{m-1}(2i+1)a_i^{(m-1)}
\left(\sinh((2i+2)x)-\sinh(2ix)\right), 
\end{align*}
and  
\begin{align*}
& \left(\sinh x\,\frac{d}{dx}\right)^{2m}\sinh x\\
& \ =\sum_{i=0}^{m-1}(2i+1)a_i^{(m-1)}\bigg\{ \frac{i+1}{2}\left(\sinh((2i+3)x)-\sinh((2i+1)x)\right)\\
& \qquad \qquad  -\frac{i}{2}\left(\sinh((2i+1)x)-\sinh((2i-1)x)\right)\bigg\}\\
& \ =\frac{1}{2}\sum_{i=1}^{m}i(2i-1)a_{i-1}^{(m-1)}\sinh((2i+1)x)\\
& \qquad -\frac{1}{2}\sum_{i=0}^{m-1}(2i+1)^2a_{i}^{(m-1)}\sinh((2i+1)x)\\
& \qquad +\frac{1}{2}\sum_{i=0}^{m-2}(i+1)(2i+3)a_{i+1}^{(m-1)}\sinh((2i+1)x).
\end{align*}
Hence, using \eqref{eq-2-5}, we complete the proof by induction.
\end{proof}

Using this lemma, we obtain
\begin{equation}
g_m(x)=\sum_{i=0}^m (2i+1)a_i^{(m)}\cosh((2i+1)x).  \label{eq-2-7}
\end{equation}

Secondly, we compute $h_m(y)$. Again by using Lemma~\ref{Lem-2-2}, we have
\begin{align}
h_m(y)&=\left(\frac{\partial}{\partial x}\right)^{2m}F(x,y)\,\bigg|_{x=0}\notag\\
      &=2\sum_{n=0}^\infty \left(\frac{d}{dx}\right)^{2m+1}\left(\tanh^{2n+1}(x/2)\right) 
      \cosh((2n+1)y)\,\bigg|_{x=0} \notag\\
      & =2\sum_{n=0}^{m} \left(\frac{d}{dx}\right)^{2m+1}\tanh^{2n+1}(x/2)\,\bigg|_{x=0}\cdot \cosh((2n+1)y) \label{eq-2-8}
\end{align}
because 
$$\tanh^{2n+1}(x/2)=\frac{x^{2n+1}}{2^{2n+1}}+O(x^{2n+2})\ \ (x\to 0).$$
We write down the right-hand side of \eqref{eq-2-8} by using the following lemma. 

\begin{lemma}\label{Lem-2-4} For $n,l\in \mathbb{Z}_{\geq 0}$, there exist sequences 
$\{b_j^{(n,l)}\}_{0\leq j\leq l}\subset \mathbb{Q}$ such that
\begin{equation}
\left(\frac{d}{dx}\right)^{l}\tanh^{2n+1}(x/2)=\sum_{j=0}^{l}b_j^{(n,l)}
\tanh^{2n+1-l+2j}(x/2), \label{eq-2-9}
\end{equation}
where $b_j^{(n,l)}=0$ if $2n+1-l+2j<0$. 
In particular, 
\begin{equation}
\left(\frac{d}{dx}\right)^{2m+1}\tanh^{2n+1}(x/2)\,\bigg|_{x=0}=b_{m-n}^{(n,2m+1)}. \label{eq-2-10}
\end{equation}
\end{lemma}

\begin{proof}
For each $n$, we can immediately obtain the form \eqref{eq-2-9} by induction on $l$, using the relation
$$\frac{d}{dx}\tanh^{2n+1}(x/2)=\frac{2n+1}{2}\left(\tanh^{2n}(x/2)-\tanh^{2n+2}(x/2)\right).$$
\end{proof}

Combining this lemma and \eqref{eq-2-8}, we obtain
\begin{equation}
h_m(y)=2\sum_{n=0}^{m}b_{m-n}^{(n,2m+1)}\cosh((2n+1)y). \label{eq-2-11}
\end{equation}

Now we are going to show $2b_{m-n}^{(n,2m+1)}=(2i+1)a_i^{(m)}$,
which implies $g_m(x)=h_m(x)$. For $m,n\in \mathbb{Z}_{\geq 0}$ with $n\leq m$, set 
$\bt_n^{(m)}=2b_{m-n}^{(n,2m+1)}$. Then, by \eqref{eq-2-10}, we have $\bt_0^{(0)}=1$. 
Furthermore the following lemma holds.

\begin{lemma}\label{Lem-2-5}\ For $m\in \mathbb{Z}_{\geq 1}$, we have the recursion
\begin{equation}
\bt_n^{(m)}=\frac{2n+1}{2}\left\{n\bt_{n-1}^{(m-1)}-(2n+1)\bt_n^{(m-1)}+(n+1)
\bt_{n+1}^{(m-1)}\right\}\quad (n\leq m),  \label{eq-2-12}
\end{equation}
where we interpret $b_{i}^{(k)}=0$ for $i<0$ or $i>k$. 
\end{lemma}

\begin{proof}
It follows from~\eqref{eq-2-9} that
\begin{equation}\label{eq3-14}
\left(\frac{d}{dx}\right)^{2m+1}\tanh^{2n+1}(x/2)=\sum_{j=0}^{2m+1}b_j^{(n,2m+1)}
\tanh^{2n-2m+2j}(x/2).
\end{equation}
Differentiating twice and using~\eqref{eq-2-9}, we see that the left-hand side is equal to
\begin{align*}
& \left(\frac{d}{dx}\right)^{2m}\left(\frac{2n+1}{2}\tanh^{2n}(x/2)-\tanh^{2n+2}(x/2)\right)\\
& =\frac{2n+1}{2}\left(\frac{d}{dx}\right)^{2m-1}\bigg\{ n\tanh^{2n-1}(x/2)-(2n+1)\tanh^{2n+1}(x/2)
+(n+1)\tanh^{2n+3}(x/2)\bigg\}\\
& =\frac{2n+1}{2}\bigg\{ n\sum_{j=0}^{2m-1}b_j^{(n-1,2m-1)}\tanh^{2n-2m+2j}(x/2)\\
& \qquad\qquad  -(2n+1)\sum_{j=0}^{2m-1}b_j^{(n,2m-1)}\tanh^{2n-2m+2+2j}(x/2)\\
& \qquad\qquad +(n+1)\sum_{j=0}^{2m-1}b_j^{(n+1,2m-1)}\tanh^{2n-2m+4+2j}(x/2)\bigg\}.
\end{align*}
If we let $x\to 0$, this goes to
\begin{align*}
& \frac{2n+1}{2}\bigg\{ nb_{m-n}^{(n-1,2m-1)}  -(2n+1)b_{m-n-1}^{(n,2m-1)}+(n+1)b_{m-n-2}^{(n+1,2m-1)}\bigg\}\\
& \ =\frac{2n+1}{4}\bigg\{ n\bt_{n-1}^{(m-1)}  -(2n+1)\bt_{n}^{(m-1)}+(n+1)\bt_{n+1}^{(m-1)}\bigg\}.
\end{align*} 
On the other-hand, the right-hand side of equation~\eqref{eq3-14} 
tends to $b_{m-n}^{(n,2m+1)}=\bt_n^{(m)}/2$ as $x\to 0$. Thus we obtain \eqref{eq-2-12}.
\end{proof}

\begin{proof}[Proof of Theorem \ref{Th-2-1}]  For $\{a_i^{(m)}\}$ defined by \eqref{eq-2-5}, 
set $\at_{i}^{(m)}=(2i+1)a_i^{(m)}$. Then \eqref{eq-2-5} can be written as $\at_0^{(0)}=1$ and 
\begin{equation*}
\at_i^{(m)}=\frac{2i+1}{2}\left\{i\at_{i-1}^{(m-1)}-(2i+1)^2\at_i^{(m-1)}+(i+1)\at_{i+1}^{(m-1)}\right\}
\end{equation*}
which has exactly the same form as \eqref{eq-2-12} for $\bt_n^{(m)}$, namely $\at_n^{(m)}=\bt_n^{(m)}$. Comparing \eqref{eq-2-7} and \eqref{eq-2-11}, we obtain $g_m(x)=h_m(x)$. Thus we complete our second
proof of Theorem \ref{Th-2-1}.
\end{proof}

\

\section{Multi-index case}\label{sec-3}

We may define the multi-poly-cosecant numbers ${D}_n^{(\kk)}$ by

\begin{equation*}
\frac{\ath(\kk; \tanh(t/2))}{\sinh t}=\sum_{n=0}^\infty {D}_n^{(\kk)}\frac{t^n}{n!},
\end{equation*}
where the function
\begin{equation*}
   \ath({\kk};z)=2^r\sum_{\substack{0<m_1<\cdots<m_r \\ m_i\equiv i \; \text{mod } 2}} \frac{z^{m_r}}{m_1^{k_1}\cdots m_r^{k_r}}
\end{equation*}
for $k_1, \ldots, k_r \in \mathbb{Z}$ is $2^r$ times $\mathrm{Ath}(\kk;z)$ which was introduced in \cite[\S5]{KT-ASPM}.
(Our $\ath_k(z)$ is $\ath(k;z)$.) 
We can regard $D_n^{(k_1,\ldots,k_r)}$ as a level 2-version of the 
multi-poly-Bernoulli numbers $B_n^{(k_1,\ldots,k_r)}$ and 
$C_n^{(k_1,\ldots,k_r)}$ defined in \cite{IKT2014}.

In \cite{KT-ASPM}, we introduced the function
\[ 
\psi(k_1,\ldots,k_r;s)=\frac{1}{\Gamma(s)}\int_0^\infty t^{s-1}\frac{\ath(k_1,\ldots,k_{r};\tanh (t/2))}{\sinh(t)}\,dt \quad (\Re s>0),
\]
which can be analytically continued to $\C$ as an entire function.  In the same manner as in the ``level 1'' case
($\xi$- and $\eta$-functions reviewed in the same paper), we see that the numbers ${D}_n^{(\kk)}$ appear 
as special values of $ \psi(k_1,\ldots,k_r;s)$ at non-positive integer arguments:
\[ \psi({\kk};-n)=(-1)^n{D}_n^{(\kk)} \quad (n=0,1,2,\ldots). \]

Also, we can obtain a similar recurrence relation for multi-poly-cosecant numbers as
\begin{equation*}
D_{n}^{(k_1,\ldots,k_{r-1},k_r-1)}=\sum_{m=0}^{\lfloor\frac{n}{2}\rfloor}\binom{n+1}{2m+1} D_{n-2m}^{(\kk)}
\end{equation*}
for any $ r\ge 1, k_i\in\Z$  and $n\ge 0 $.


\ 

{\bf Acknowledgements.}\ 
{This work was supported by Japan Society for the Promotion of Science, Grant-in-Aid for Scientific 
Research (S) 16H06336 (M. Kaneko), and (C) 18K03218 (H. Tsumura).}

\

\end{document}